\documentclass{article}
\usepackage{amsmath,amssymb,amsthm,amsfonts}
\usepackage{graphicx}
\usepackage{authblk,hyperref}
\usepackage{multicol,multirow}
\usepackage{mathrsfs}
\usepackage{cleveref}
\usepackage{xcolor}
\usepackage{enumitem}
\newtheorem*{theorem*}{Theorem}
\newtheorem{theorem}{Theorem}[section]
\newtheorem{lemma}[theorem]{Lemma}
\newtheorem{definition}[theorem]{Definition}
\newtheorem{remark}[theorem]{Remark}

\newtheorem*{conjecture*}{Conjecture}
\newtheorem{conjecture}[theorem]{Conjecture}
\newtheorem*{question*}{Question}
\newtheorem{question}[theorem]{Question}
\newtheorem{corollary}[theorem]{Corollary}
\numberwithin{theorem}{section}
\numberwithin{equation}{section}

\begin{document}
\title{T-Fermat integers}
\author[1]{Tigran Hakobyan}
\affil[1]{\small{American University of Armenia, Armenia, \textit{thakobyan@aua.am}}}
\footnotetext[1]{This work was supported by the Higher Education and Science Committee of RA (Research
Project No 24RL-1A028)}
\maketitle
\abstract{
We introduce the concept of a T-Fermat integer, which generalizes the notion of a prime number. We show that any composite T-Fermat integer, if one exists, must be a Carmichael number. We prove several properties of T-Fermat integers and conjecture that there are infinitely many composite T-Fermat integers. Together with a further structural conjecture, this suggests a possible route toward proving the infinitude of primes $p$ such that $\omega(p-1)\leq 2$.
}

\section{Introduction}

Fermat's little theorem states that $a^p\equiv a(\text{mod}\ p)$ for every integer $a$ and every prime $p$. One might ask whether the converse holds.
\begin{question*}
\textit{If $n>1$ is a positive integer such that $a^n\equiv a(\text{mod}\ n)$ for every integer $a$, must $n$ be prime?}
\end{question*}
In 1910, Carmichael \cite{carmichael1910note}, using Korselt's criterion \cite{Korselt1899}, gave the counterexample $n=561=3\times11\times 17$, thus resolving the question negatively. This led him to introduce Carmichael numbers as composite numbers satisfying the aforementioned congruence. 

A significant amount of research has been carried out on Carmichael numbers. Specifically, in \cite{carmichael1912composite} Carmichael conjectured the infinitude of these numbers. This was proved only in 1994 by Alford, Granville, and Pomerance \cite{alford1994there}. In particular, they established the inequality $C(x)> x^{2/7}$ for sufficiently large $x$ where $C(x)$ denotes the number of Carmichael numbers up to $x.$ Another result in this direction, due to Banks and Pomerance \cite{banks2010carmichael}, is in the spirit of Dirichlet’s theorem on arithmetic progressions. Notably, the authors prove that there are infinitely many Carmichael numbers in a given arithmetic progression $(a+lm)_{l\geq 1}$ with $\gcd(a,m)=1.$ A refinement of this work was obtained by D. Larsen \cite{larsen2025carmichael}, who proved that any arithmetic progression either contains no Carmichael numbers or infinitely many.

In this paper, we take a different approach by introducing a new class of integers that generalizes the primes. First, note that Fermat's little theorem can be rewritten in a slightly different form: 
\begin{equation*}
    p|(x^p+x)-2x \quad\text{for any }x\in\mathbb{Z}.
\end{equation*}
 This observation motivated us to introduce the following two definitions.

\begin{definition}
    For a positive integer $n$ we define the polynomial $$T_n(x)=\sum_{d|n}x^d-d(n)x.$$
\end{definition}
\begin{definition}
    A positive integer $n>1$ is called a \textbf{weakly T-Fermat integer} if $$n|T_n(x)$$ for every integer $x$. If, in addition, $n$ is square-free, then $n$ is called a \textbf{T-Fermat integer}. 
\end{definition}
Clearly, any prime is a T-Fermat integer. Meanwhile, an easy check shows that the number $4$ is a weakly T-Fermat integer, yet not a prime. This naturally leads to the following 
\begin{question}\label{ExistAlmostPrime}
Are there any composite weakly T-Fermat integers greater than four? Are there any composite T-Fermat integers? 
\end{question}
In this paper, we focus on T-Fermat integers. Initially, we believed that no composite T-Fermat integers existed, and we attempted to prove this using the following straightforward argument: pick any prime divisors $p$ and $q$ of a T-Fermat integer $n$ and show that $p-1$ and $q-1$ share the same divisors. This turns out to be true for certain special divisors, such as powers of two and Fermat primes. Specifically, we prove the following

\begin{theorem}\label{powerOf_two_and_Fermat}
Let $n$ be a T-Fermat integer. Then
\begin{enumerate}
\item The value $\nu_2(p-1)$ is independent of the choice of prime divisor $p$ of $n.$ 
\item  If a Fermat prime $r$ divides $p-1$ for some prime divisor $p$ of $n$, then $r|q-1$ for every prime divisor $q$ of $n.$
\end{enumerate}
\end{theorem}
For a general divisor, however, the situation becomes more delicate and leads to substantially more complicated conditions. The second assertion of Theorem \ref{powerOf_two_and_Fermat} was also posed as a problem at the 2024 International Mathematics Competition for University Students \cite{IMC}.

The next crucial turning point came when we discovered that any composite T-Fermat integer must be a Carmichael number.

\begin{theorem}\label{AlmostPrimeIsCarmichael}
Any T-Fermat integer is either prime or a Carmichael number.
\end{theorem}

This result shifted the focus of our investigation toward the search for composite T-Fermat integers. In particular, we aim to view the set of composite T-Fermat integers as a special  
subset of Carmichael numbers. Thus, the following strengthened version of Question \ref{ExistAlmostPrime} arises naturally.
\begin{conjecture}\label{InfinitelyManyComposite}
There are infinitely many composite T-Fermat integers.  
\end{conjecture}
In trying to identify a T-Fermat integer, we sought to translate the T-Fermat integer condition into the language of polynomials, algebraic number theory, and linear algebra. These areas provide a rich toolkit of advanced methods, allowing us to establish the following key result:

\begin{theorem}\label{MainTheoremSingle}
     Let $n$ be a T-Fermat integer, let $p|n$ be prime, and let $s=q^{\nu}$ be a divisor of $p-1$, where $q$ is an odd prime. Let $g$ be a generator of the group $(\mathbb{Z}/s\mathbb{Z})^{*}$, and for each residue class $r\in (\mathbb{Z}/s\mathbb{Z})^{*}$ let $\omega_r$ denote the number of prime divisors of $n$ congruent to $r$ modulo $s$.
     Let $K=\mathbb{Q}(\zeta),$ where $\zeta$ is a primitive $\varphi(s)$-th root of unity and let $\mathcal{O}_K=\mathbb{Z}[\zeta]$ be its ring of integers. Then the congruence $f(\eta)\equiv d(n)(\text{mod}\ p\mathcal{O}_K)$ holds for each $\varphi(s)$-th root of unity $\eta$, where $f(x)=\displaystyle\prod_{\ell=1}^{\varphi(s)}\left(1+x^\ell\right)^{\omega_{g^\ell}}.$
\end{theorem}

Several useful corollaries follow from Theorem \ref{MainTheoremSingle}.

\begin{corollary}\label{FermatPrimeRefined}
    If $\nu_2(\ell)<\nu_2(\varphi(s))$, then $\omega_{g^\ell}=0$. In particular, if $s=q$ is a Fermat prime and $q|p_0-1$ for some prime divisor $p_0$ of $n$, then $p\equiv 1(\text{mod}\ q)$ for every prime divisor $p$ of $n$. 
\end{corollary}

\begin{remark}
Corollary \ref{FermatPrimeRefined} refines the second assertion of Theorem \ref{powerOf_two_and_Fermat}.
\end{remark}

\begin{corollary}\label{PowerOfTwoCongruence}
With the notation of Theorem \ref{MainTheoremSingle}, define $$A(t):=\sum_{\ell=1}^{\varphi(s)}\left(t-\gcd(t,\ell)\right)\omega_{g^\ell}$$ and $$B:=\varphi(\varphi(s))(\omega(n)-\omega_1).$$ 
The following congruences are valid     
    \begin{enumerate}
        \item  $2^{A(t)}\equiv 1(\text{mod}\ p)$ for each divisor $t$ of $\varphi(s)$
        \item $2^B\equiv 1(\text{mod}\ p)$ 
   \end{enumerate}
\end{corollary}

\begin{corollary}\label{OmegaOneSingle}
 Let $n$ be a T-Fermat integer and let $p=1+2^{h}q^{\nu}$ be a prime divisor of $n$, where $h\leq 5, \nu\geq 1$, and $q$ is an odd prime. Then  $$\omega(n)\equiv \omega_1(\text{mod}\ q).$$  In particular, either $\omega(n)>q$ or every divisor $d$ of $n$ satisfies $d\equiv 1(\text{mod}\ q)$.
\end{corollary}

Observe that Theorem \ref{MainTheoremSingle} and Corollaries \ref{FermatPrimeRefined}, \ref{PowerOfTwoCongruence}, and \ref{OmegaOneSingle} are dealing with divisors of the form $s=q^\nu$ only. Trying to extend these results to the case of arbitrary divisors, we came up with the following ``multidimensional'' generalizations.

\begin{theorem}\label{MainTheoremMultiple}
    Let $n$ be a T-Fermat integer, let $p$ be a prime divisor of $n$, and let $s=s_1s_2...s_r$ be a divisor of $p-1$, where $s_i=q^{\nu_i}_i$ for distinct odd primes $q_i$. For each $i$, let $g_i$ be a generator of $(\mathbb{Z}/s_i\mathbb{Z})^{*}$. For $L=(\ell_1,\ell_2,...\ell_r)$, let $\omega_L$ denote the number of prime divisors of $n$ congruent to $g^{\ell_j}_j$ modulo $s_j$ for every $1\leq j\leq r$. Let $K=\mathbb{Q}(\zeta),$ where $\zeta$ is a primitive $\varphi(s)$-th root of unity and let $\mathcal{O}_K=\mathbb{Z}[\zeta]$ be its ring of integers. Then the congruence $f(\eta_1,\eta_2,...,\eta_r)\equiv d(n)(\text{mod}\ p\mathcal{O}_K)$ holds for any $\varphi(s_i)$-th roots of unity $\eta_i, 1\leq i\leq r$, where 
$$f(x_1,x_2,...,x_r)=\displaystyle\prod_{\ell_1=1}^{\varphi(s_1)}\cdots \prod_{\ell_r=1}^{\varphi(s_r)}\left(1+\prod_{i=1}^r x^{\ell_i}_i\right)^{\omega_{L}}.$$
\end{theorem}

\begin{corollary}\label{OmegaOneMultiple}
With the notation of Theorem \ref{MainTheoremMultiple}, define $$A(t_1,...,t_r):=\displaystyle\sum_{\ell_1=1}^{\varphi(s_1)}\cdots \sum_{\ell_r=1}^{\varphi(s_r)}\left(\prod_{i=1}^r t_i-\prod_{i=1}^r \gcd(t_i,\ell_i)\right)\omega_L$$
and let $t_i|\varphi(s_i)$ for each $1\leq i\leq r$. Then
$$2^{A(t_1,...,t_r)}\equiv 1(\text{mod}\ p).$$
\end{corollary}

\begin{corollary}\label{PowerOfTwoCongruenceMultiple}
In the notation of Theorem \ref{MainTheoremMultiple} assume $r=2$. Let $d=\gcd(\varphi(s_1),\varphi(s_2))$ and let $x_0$ be determined modulo $s=s_1s_2$ by the congruences $x_0\equiv g^{\frac{\varphi(s_1)}{d}}_1(\text{mod}\ s_1)$ and $x_0\equiv g^{-\frac{\varphi(s_2)}{d}}_2(\text{mod}\ s_2)$. We also define 
\begin{equation*}
B:=\varphi(\varphi(s))\left(\omega(n)-\sum_{k=0}^{d-1}\omega_{x^k_0}\right).
\end{equation*}
Then 
$$2^{B}\equiv 1(\text{mod}\ p).$$ 
\end{corollary}

Theorem \ref{MainTheoremMultiple} and Corollaries
\ref{OmegaOneMultiple} and \ref{PowerOfTwoCongruenceMultiple} 
impose strong restrictions on the prime divisors $p$ satisfying the condition $\omega(p-1)>2$. Thus, it is natural to expect that the number of such primes in the prime factorization of $n$ is relatively small. This observation motivated us to formulate the following 

\begin{conjecture}\label{SmallPortionOfPrimes}
There exists a constant $0<\epsilon<1$ such that $$\# \left\{p|n\ |\ \omega(p-1)>2\right\}<\epsilon \omega(n)$$ for infinitely many composite T-Fermat integers $n$. 
\end{conjecture}

If Conjecture \ref{SmallPortionOfPrimes} holds, it yields the following significant result. 
\begin{theorem*}
There exist infinitely many primes $p$ such that $p-1$ has at most two prime divisors. 
\end{theorem*}

\begin{proof} Let $M$ be a positive integer. By Corollary \ref{fin_many_with_few_primes}, there exists $N$ such that every composite T-Fermat integer $n>N$ satisfies $\omega(n)>M$. Now choose a composite T-Fermat integer $n>N$ satisfying Conjecture \ref{SmallPortionOfPrimes}. Then $$\# \left\{p|n\ |\ \omega(p-1)\leq 2\right\}>(1-\epsilon)\omega(n)>(1-\epsilon)M.$$ Since $M$ is arbitrary, the set of primes $p$ such that $\omega(p-1)\le 2$ is infinite. \end{proof}

The paper is organized as follows. In Section \ref{AuxResults}, we prove several auxiliary results on T-Fermat integers, as well as Theorems \ref{powerOf_two_and_Fermat} and \ref{AlmostPrimeIsCarmichael}. In Section \ref{sec:MainTheoremSingle}, we prove Theorem \ref{MainTheoremSingle} together with Corollaries \ref{FermatPrimeRefined}, \ref{PowerOfTwoCongruence}, and \ref{OmegaOneSingle}. In Section \ref{sec: MainTheoremMultiple}, we prove Theorem \ref{MainTheoremMultiple} and Corollaries \ref{OmegaOneMultiple} and \ref{PowerOfTwoCongruenceMultiple}.
We use the following notation throughout the paper.

\begin{itemize}
\item $\gcd(m,n)$ - the greatest common divisor of $m$ and $n$. 
\item $d(n)$ - the number of positive divisors of $n.$
\item $\omega(n)$ - the number of distinct prime divisors of $n.$
\item $\omega_r=\omega_r(n,s)$ - the number of distinct prime divisors of $n$ that are congruent to $r$ modulo the given integer $s.$ We simply write $\omega_r$ when $n$ and $s$ are clear from the context. 
\item $\varphi(n)$ - Euler's totient function.
\item $\nu_p(n)$ - the exponent of $p$ in the prime factorization of $n.$ 
\item  $A\otimes B$ - Kronecker (tensor) product of matrices $A$ and $B$. For multiple matrices we use                 $\displaystyle\bigotimes_{i=1}^r A_i$.
\end{itemize}
\section{Auxiliary results}\label{AuxResults}
\begin{lemma}\label{KeyLemma}
    Let $n$ be a weakly T-Fermat integer and let $p$ be a prime divisor of $n.$ If $m|p-1$ and $r\not\equiv 1 (\text{mod}\ m),$ then the number of positive divisors of $n$ congruent to $r$ modulo $m$ is divisible by $p.$
\end{lemma}
\begin{proof}
    For each $0\leq i\leq p-2$ let $h_i$ denote the number of positive divisors of $n,$ congruent to $i$ modulo $p-1,$ and similarly for each $0\leq j\leq m-1$ let $\nu_j$ denote the number of positive divisors of $n,$ congruent to $j$ modulo $m.$ Since $p|n$ and $n$ is a weakly T-Fermat integer, the polynomial $T_n(x)$ vanishes on $\mathbb{F}_p.$ On the other hand, $$T_n(x)=(h_1-k)x+\sum_{i\neq 1}h_i x^i$$ in $\mathbb{F}_p[x],$ so $p|h_i$ for all $0\leq i\leq p-2, i\neq 1.$ It follows that $$\nu_r=h_r+h_{r+m}+h_{r+2m}+...\equiv 0(\text{mod}\ p)$$ for each $r\not\equiv 1(\text{mod}\ m).$    
\end{proof}
\begin{corollary}\label{PrimeDivisorBound}
    Any prime divisor of a composite T-Fermat integer $n$ is smaller than  $d(n).$
\end{corollary}
\begin{proof}
    Let $p$ and $q$ be the smallest and the largest prime divisors of $n$, respectively. Since $p\not\equiv 1(\text{mod}\ q-1),$ Lemma \ref{KeyLemma} implies that there are at least $q$ divisors of $n$ congruent to $p$ modulo $q-1.$ Hence $q<d(n)$, proving the claim.
\end{proof}

\begin{corollary}\label{fin_many_with_few_primes}
    For any number $M$ there are only finitely many composite T-Fermat integers having $M$ prime divisors.
\end{corollary}
\begin{proof}
    If $n$ is a composite T-Fermat integer with $\omega(n)=M$, then any prime divisor of $n$ is less than $d(n)$ by Corollary \ref{PrimeDivisorBound}. Therefore, 
     $$n<d^M(n).$$ On the other hand, it is known that $d(n)=o(n^\varepsilon)$ for any $\varepsilon>0$ 
    \cite[Theorem~315, Ch.~18]{hardy1968introduction}. Combining these two facts
    proves the result. 
\end{proof}
\begin{corollary}\label{CoprimeWithGcd}
    If $n$ is a T-Fermat integer, then $\gcd(n,\varphi(n))=1.$
\end{corollary}
\begin{proof}
Assume to the contrary that $\gcd(n,\varphi(n))>1.$ Then there exist primes $p$ and $q$ dividing $n$ such that $p\equiv 1(\text{mod}\ q).$ It follows from Lemma \ref{KeyLemma} with $m=q$ that $2^{\omega(n)-1}=\nu_0$ is divisible by $p,$ a contradiction. 
\end{proof}
\begin{corollary}\label{T-Fermat is odd}
    A T-Fermat integer $n>2$ is odd.
\end{corollary}
\begin{proof}
    This follows immediately from Corollary \ref{CoprimeWithGcd}.
\end{proof}

\begin{lemma}\label{AuxLemma}
    Let $p$ be a prime number, let $h$ be a positive integer coprime to $p-1,$ and let $f(x)=\sum_{i=0}^{\ell-1}(-1)^ix^{h^i}$. If $\ell$ is the order of $h$ modulo $p-1,$ then there exists $a\in\mathbb{F}_p$ such that $a^{h^\ell}=a$ and $f(a)\neq 0.$
\end{lemma}
\begin{proof}
    Observe that $a^{h^\ell}=a$ for any $a\in\mathbb{F}_p$ since $p-1|h^\ell-1.$ On the other hand, the numbers $h^0,h^1,...,h^{\ell-1}$ have different remainders upon division by $p-1.$ Consequently, the polynomial $f(x)$ does not vanish on $\mathbb{F}_p$. This proves the existence of an element with the required properties. 
\end{proof}
\begin{lemma}\label{OrderIsOdd}
    If $n$ is a T-Fermat integer, then for any primes $p$ and $q$ dividing $n,$ the order of $q$ modulo $p-1$ is an odd number. 
\end{lemma} 
\begin{proof}
    By Corollary \ref{CoprimeWithGcd} the order $\ell$ of $q$ modulo $p-1$ is well defined. We assume to the contrary that $\ell$ is an even number. According to Lemma \ref{AuxLemma} with $h=q$, there exists $a\in\mathbb{F}_p$ such that $a^{q^\ell}=a$ and  $f(a)=\sum_{i=0}^{l-1}(-1)^ia^{q^i}\neq 0$. We consider the sequence $(a_i)_{i=0}^\ell\subset\mathbb{F}_p$ defined by $a_0=a$ and $a_{i+1}=-a^q_i$ for $0\leq i\leq \ell-1.$ Since $\ell$ is even by the assumption, we have $a_{\ell}=a^{q^\ell}_0=a_0.$ It follows that

    \begin{align*}
    &\sum_{i=0}^{\ell-1}\sum_{d|n}a^d_i=\sum_{i=0}^{\ell-1}\left(\sum_{d|\frac{n}{q}}a^d_i+\sum_{d|\frac{n}{q}}a^{qd}_i\right)=\sum_{i=0}^{\ell-1}\left(\sum_{d|\frac{n}{q}}a^d_{i+1}+\sum_{d|\frac{n}{q}}a^{qd}_i\right)=\\ &=\sum_{d|\frac{n}{q}}\sum_{i=0}^{l-1}\left(a^d_{i+1}+a^{qd}_i\right)=\sum_{d|\frac{n}{q}}\sum_{i=0}^{l-1}\left(a^d_{i+1}-a^d_{i+1}\right)=0,\\
    \end{align*}
     as every $d|n$ is odd by Corollary \ref{T-Fermat is odd}. Since $p|n$ and $n$ is a T-Fermat integer, we have $T_n(a_i)\equiv 0(\text{mod} \ p)$, hence $$\displaystyle\sum_{d|n}a^d_i=d(n)a_i$$ in $\mathbb{F}_p$. This gives, $$d(n)f(a)=d(n)\sum_{i=0}^{\ell-1}a_i=\sum_{i=0}^{\ell-1}d(n)a_i=\sum_{i=0}^{\ell-1}\sum_{d|n}a^d_i=0$$ in $\mathbb{F}_p$ which is impossible. Indeed, $f(a)\neq 0$ by construction, and $d(n)=2^{\omega(n)}$ is coprime to $p.$ The attained contradiction finishes the proof.
\end{proof}

\subsection{Proof of Theorem 1.4}

\begin{proof}

\noindent (a) Let $p,q|n$ be prime divisors. By Lemma \ref{OrderIsOdd}, there exists an odd positive integer $\ell$ such that $$q^\ell\equiv 1(\text{mod}\ p-1).$$ Since $\displaystyle q^\ell-1=(q-1)\sum_{i=0}^{\ell-1}q^i$ and the second factor is odd, we obtain $$\nu_2(q-1)\geq \nu_2(p-1).$$ By symmetry, equality holds.
\vskip 2pt
\noindent (b) Let $r$ be a Fermat prime such that $r|p-1$ for some prime divisor $p|n$, and let $q$ be any prime divisor of $n$. By Lemma \ref{OrderIsOdd}, $q^\ell\equiv 1(\text{mod}\ p-1)$ with $\ell$ odd, so $$q^\ell\equiv 1(\text{mod}\ r).$$ Consequently, $$q=q^{\gcd(\ell,r-1)}\equiv 1(\text{mod}\ r).$$

\end{proof}

\subsection{Proof of Theorem 1.5}
\begin{proof}
    Let $p$ be a prime divisor of the composite T-Fermat integer $n$.  By Lemma \ref{KeyLemma}, for every residue class modulo $p-1$ other than the class of 1, the number of positive divisors of $n$ contained in that class is divisible by $p$. Since $d(n)=2^{\omega(n)}$ is coprime to $p$, the number of divisors of $n$ congruent to $1(\text{mod}\ p-1)$ is not divisible by $p$. By Corollary \ref{CoprimeWithGcd}, $\gcd(d,p-1)=1$ for every divisor $d$ of $n$, so the condition $d\equiv n(\text{mod}\ p-1)$ is equivalent to the condition $\frac{n}{d}\equiv 1(\text{mod}\ p-1)$. It follows that the number of divisors of $n$ congruent to $n$ modulo $p-1$ is equal to the number of divisors congruent to 1 modulo $p-1.$ Since the latter is not divisible by $p$, we infer that $n\equiv 1(\text{mod}\ p-1)$ is the only possibility. Consequently, $n$ is a composite square-free number such that $p-1|n-1$ for every prime  $p|n$. Applying Korselt's criterion, we deduce that $n$ is a Carmichael number.
\end{proof}

\section{Proof of Theorem 1.7 and corollaries}\label{sec:MainTheoremSingle}

\subsection{Proof of Theorem 1.7}

\begin{proof}
    Let $1=r_1<r_2<\ldots<r_k=s-1$ be the complete list of residues coprime to $s$, where $k=\varphi(s)$. For each positive integer $h$ we define the $k\times 1$ vector $$\delta_h=\left[h_{r_1}, h_{r_2},\dots,h_{r_k}\right]^T,$$ where $h_{r_i}$ is the number of positive divisors of $h$ congruent to $r_i$ modulo $s$. Observe that for any integer $a$ coprime to $h,$ $$\delta_{ah}=M_a\delta_h$$ for the $k\times k$ matrix $M_a=(m_{ij}),$ where $m_{ij}=a_{r^{-1}_jr_i}, 1\leq i,j\leq k$. If, in addition, $a$ is a prime number coprime to $s,$ then $$M_a=I+P_a,$$ where $P_a$ is the column permutation matrix, corresponding to the permutation $\sigma$ given by $r_{\sigma(i)}\equiv ar_i (\text{mod}\ s)$ for all $1\leq i\leq k$. That is, $P_a=(p_{ij})$ is the matrix defined by
    $$p_{ij}=\begin{cases} 1, & r_i\equiv ar_j(\text{mod}\ s)\\
    0, & \text{otherwise}\end{cases}$$
    Moreover, if $\ell\in\{1,2,\ldots,k\}$ is the unique integer such that $a\equiv g^\ell(\text{mod}\ s)$, then $M_a=I+P^\ell_g,$ where this time $P_g$ is the column permutation matrix corresponding to the action of $g$ on the group $(\mathbb{Z}/s\mathbb{Z})^{*}.$ By Corollary \ref{CoprimeWithGcd}, $n$ is coprime to $\varphi(n)$ and is therefore coprime to $s.$ Using the equality $\delta_{ph}=M_p\delta_h$ several times, we get $$\delta_n=\left(\prod_{p|n}M_p\right)\delta_1=\left(\prod_{\ell=1}^{k}\left(I+P^\ell_g\right)^{\omega_{g^\ell}}\right)e_1=f(P_g)e_1,$$ where $e_1$ is the first standard basis vector. On the other hand, Lemma \ref{KeyLemma} implies that $\delta_n=d(n)e_1 (\text{mod}\ p)$, so $$f(P_g)e_1=d(n)e_1+pv$$ for some vector $v\in\mathbb{Z}^{k}.$ Since $g$ is a generator of $(\mathbb{Z}/s\mathbb{Z})^{*},$ it follows that the corresponding permutation is a cycle and the characteristic polynomial of $P_g$ is $t^{k}-1.$ Unitarily diagonalizing $P_g$, we get $P_g=Q\Lambda Q^{-1},$ where  $$\Lambda=\mathrm{diag}(1,\zeta,...,\zeta^{k-1}),$$ $Q$ is unitary, and the $j$th column of $Q$ has the form $$\frac{1}{\sqrt{k}}\left[\zeta^{n_{1}(j-1)},\zeta^{n_{2}(j-1)},\ldots,\zeta^{n_{k}(j-1)}\right]^T$$ for a suitable permutation $\{n_1,n_2,...,n_k\}$ of $\{1,2,...,k\}$ depending on $j$. Let $v_0$ be the first column of $Q^{-1}$. Combining the equalities $P_g=Q\Lambda Q^{-1}, Q^{*}=Q^{-1},$ and $ f(P_g)e_1=d(n)e_1+pv$, we get $$\left(f(\Lambda)-d(n)I\right)\left(\sqrt{k}v_0\right)=p\left(\sqrt{k}Q^{*}v\right).$$Since all the entries of $\sqrt{k}v_0$ are roots of unity and $\sqrt{k}Q^{*}v\in\mathbb{Z}[\zeta]^k,$ we obtain $$f(\eta)-d(n)\in p\mathcal{O}_K$$ for any $k=\varphi(s)$th root of unity $\eta.$ 
\end{proof}

\subsection{Proof of Corollary 1.8}\label{Proof of Cor 1.8:subsec}

\begin{proof}
    The condition $\nu_2(\ell)<\nu_2(\varphi(s))$ implies that $\ell=2^tb$, where $t<m=\nu_2(\varphi(s))$ and $b$ is odd. Since $2^{t+1}|\varphi(s)$, there is a primitive $2^{t+1}$th root of unity $\eta$ among the $\varphi(s)$th roots of unity. Since $1+\eta^{2^t}=0$ and $l=2^tb$ with $b$ odd, it follows that $1+\eta^\ell=0$. If $\omega_{g^\ell}>0$, we obtain from Theorem \ref{MainTheoremSingle} the congruence $$2^{\omega(n)}=d(n)\equiv f(\eta)=0(\text{mod}\ p)$$ in $\mathcal{O}_K$ which is impossible since $p$ is odd. Consequently, $\omega_{g^\ell}=0$. 
\end{proof}

\subsection{Proof of Corollary 1.10}

\begin{proof} 

 (a)  By Theorem \ref{MainTheoremSingle} $$d^t(n)\equiv \prod_{\eta^t=1}f(\eta)=\mathrm{Res}(x^t-1,f(x))=\prod_{\ell=1}^{\varphi(s)}\mathrm{Res}(x^t-1,x^\ell+1)^{\omega_{g^\ell}}(\text{mod}\ p)$$ in $\mathcal{O}_K$, where $\mathrm{Res}(f,g)$ is the resultant of $f$ and $g$. Using the identity $$\mathrm{Res}(x^u-a, x^v-b)=(-1)^v\left(a^{\frac{v}{d}}-b^{\frac{u}{d}}\right)^{\gcd(u,v)},$$ we get $$\mathrm{Res}(x^t-1,x^\ell+1)=2^{\gcd(t,\ell)}$$ for any $\ell$ satisfying the condition $\nu_2(\ell)\geq\nu_2(\varphi(s))$. Together with Corollary \ref{FermatPrimeRefined}, this yields $$d^t(n)\equiv 2^{\sum_{\ell=1}^{\varphi(s)}\gcd(t,\ell)\omega_{g^\ell}}(\text{mod}\ p).$$ Taking into account the equality $d(n)=2^{\omega(n)}=2^{\sum_{\ell=1}^{\varphi(s)}\omega_{g^\ell}},$ we obtain $$2^{\sum_{\ell=1}^{\varphi(s)}\left(t-\gcd(t,\ell)\right)\omega_{g^\ell}}\equiv 1(\text{mod}\ p\mathcal{O}_K).$$ Therefore, the same congruence is valid in $\mathbb{Z}$ as well.

(b) Taking norms on both sides of the congruence $f(\zeta)\equiv d(n)(\text{mod}\ p)$, which is valid due to Theorem \ref{MainTheoremSingle}, we get $$d(n)^{[K:\mathbb{Q}]}=N_{K/\mathbb{Q}}(d(n))\equiv N_{K/\mathbb{Q}}(f(\zeta))=\prod_{\ell=1}^{\varphi(s)}\left(N_{K/\mathbb{Q}}\left(1+\zeta^\ell\right)\right)^{\omega_{g^\ell}}(\text{mod}\ p).$$
Observe that the order of $\zeta^\ell$ is odd whenever $\nu_2(\ell)\geq\nu_2(\varphi(s))$. Furthermore, it is greater than one if and only if $\ell<\varphi(s)$. Therefore, we have the equality 
$\omega_{g^\ell}=0$ whenever $\nu_2(\ell)<\nu_2(\varphi(s))$. Moreover, if $m>1$ is an odd divisor of $\varphi(s)$, $\zeta_m$ is a primitive $m$th root of unity, and $\Phi_m$ is the $m$th cyclotomic polynomial, then $\Phi_m(-1)=1$ and $$N_{K/\mathbb{Q}}(1+\zeta_m)=\left(\Phi_m(-1)\right)^{\frac{\varphi(\varphi(s))}{\varphi(m)}}=1.$$
Finally, we obtain $$2^{[K:\mathbb{Q}]\omega(n)}=d(n)^{[K:\mathbb{Q}]}\equiv (N_{K/\mathbb{Q}}(2))^{\omega_{1}}=2^{[K:\mathbb{Q}]\omega_1}(\text{mod}\ p),$$ which is the desired result in light of the equality $[K:\mathbb{Q}]=\varphi(\varphi(s))$. 
\end{proof}  

\subsection{Proof of Corollary 1.11}

\begin{proof}
Observe that the order of $2$ modulo $p$ is divisible by $q$. Indeed, by Fermat's little theorem, $$p|2^{p-1}-1=2^{2^hq^\nu}-1.$$ If this order were not divisible by $q$, then $p$ would have to divide $$2^{2^5}-1=\prod_{k=0}^4(1+2^{2^k}),$$ which is the product of the five known Fermat primes. Since $p$ is not a Fermat prime by assumption, this is impossible. According to Theorem \ref{MainTheoremSingle} with $s=q$ we get $$2^{\varphi(q-1)(\omega(n)-\omega_1)}\equiv 1(\text{mod}\ p).$$ It follows that $q| \omega(n)-\omega_1$, as $\varphi(q-1)$ is coprime to $q$. If $q\geq \omega(n)$, we obtain $\omega_1=\omega(n)$, so all the prime divisors of $n$ are congruent to 1 modulo $q$. The conclusion follows. 
\end{proof}
\section{Proof of Theorem 1.12 and corollaries}\label{sec: MainTheoremMultiple}

\subsection{Proof of Theorem 1.12}

\begin{proof}
This proof follows the notation and strategy of the proof of Theorem  \ref{MainTheoremSingle}. Under the Chinese remainder identification $$(\mathbb{Z}/s\mathbb{Z})^*\cong \prod_{i=1}^r (\mathbb{Z}/s_i\mathbb{Z})^*,$$ multiplication by a prime corresponding to $L=(\ell_1,...,\ell_r)$ is represented by the matrix $\displaystyle\bigotimes_{i=1}^{r}  P^{\ell_i}_{g_i}$. We can thus write $$\delta_n=\left(\prod_{p|n}M_p\right)\delta_1=\left(\prod_{\ell_1=1}^{\varphi(s_1)}\cdots \prod_{\ell_r=1}^{\varphi(s_r)}\left(I+\bigotimes_{i=1}^{r}  P^{\ell_i}_{g_i}\right)^{\omega_{L}}\right)e_1.$$
Using the unitary diagonalizations $P_{g_i}=Q_i\Lambda_i Q^{-1}_i$ and the relation $\delta_n=d(n)e_1 (\text{mod}\ p)$, we get $$\left(\Lambda-d(n)I\right)\left(\sqrt{\varphi(s)}v_0\right)=p\left(\sqrt{\varphi(s)}Q^{*}v\right),$$ where $Q=\displaystyle\bigotimes_{i=1}^r Q_i, v\in\mathbb{Z}^{\varphi(s)},$ $v_0$ is the first column of $Q^{*}$, and $$\displaystyle \Lambda=\prod_{\ell_1=1}^{\varphi(s_1)}\cdots \prod_{\ell_r=1}^{\varphi(s_r)}\left(I+\bigotimes_{i = 1}^{r} \Lambda^{\ell_i}_i\right)^{\omega_{L}}.$$ Finally, note that all the  entries of $\sqrt{\varphi(s)}v_0$ are roots of unity, so
$$f(\eta_1,\eta_2,...,\eta_r)=\displaystyle\prod_{\ell_1=1}^{\varphi(s_1)}\cdots \prod_{\ell_r=1}^{\varphi(s_r)}\left(1+\prod_{i=1}^r \eta^{\ell_i}_i\right)^{\omega_{L}}=d(n)(\text{mod}\ p\mathcal{O}_K)$$
    for any $\varphi(s_i)$th roots of unity $\eta_i, 1\leq i\leq r$.
\end{proof}

\subsection{Proof of Corollary 1.13}

\begin{proof}
    For simplicity we introduce the following additional notations: 
    \begin{itemize}
    \item $X(n_1,...,n_r):=\left\{(\eta_1,...,\eta_r)\ |\ \eta^{n_j}_j=1, 1\leq j\leq r\right\}$
    \item 
    $T:=(t_1,...,t_r), L:=(\ell_1,...,\ell_r),$ and $\gcd(T,L):=\prod_{j=1}^r \gcd(t_j,\ell_j)$
    \item 
    $\mu_j:=\frac{t_j}{\gcd(t_j,\ell_j)}, 1\leq j\leq r$ and $M:=\mathrm{lcm}(\mu_1,...,\mu_r)$
    \end{itemize}
    We fix integers $\ell_1,\ell_2,...,\ell_r$ satisfying the conditions $$\nu_2(\ell_j)\geq \nu_2(\varphi(s_j)), 1\leq j\leq r.$$ Clearly, $M$ is odd and $$\prod_{\eta^M=1}(x+\eta)=x^M+(-1)^{M+1}=x^M+1.$$ Plugging in $x=1$, we get $\prod_{\eta^M=1}(1+\eta)=2$. Thus,
    \begin{align*}
    &\prod_{X(t_1,...t_r)}\left(1+\prod_{j=1}^r \eta^{\ell_j}_j\right)=\left(\prod_{X(\mu_1,...,\mu_r)}\left(1+\prod_{j=1}^r\eta_j\right)\right)^{\gcd(T,L)}=\\                  &=\left(\prod_{\eta^M=1}(1+\eta)\right)^{\gcd(T,L)}=2^{\gcd(T,L)}.\\ 
    \end{align*}
    In light of Theorem \ref{MainTheoremMultiple} and Corollary \ref{FermatPrimeRefined}, $\omega_{L}=0$ whenever $\nu_2(\ell_j)<\nu_2(\varphi(s_j))$ for some $j$. Therefore,
    \begin{align*}
    & 2^{\left(\sum_{\ell_1=1}^{\varphi(s_1)}\cdots \sum_{\ell_r=1}^{\varphi(s_r)}\omega_{L}\right)\prod_{j=1}^r t_j}=2^{\omega(n)\prod_{j=1}^r t_j}=\\ &= d(n)^{\prod_{j=1}^r t_j} 
    \equiv \prod_{X(t_1,\dots,t_r)} f(\eta_1,\eta_2,\dots,\eta_r)= \\
    &= \prod_{\ell_1=1}^{\varphi(s_1)}\cdots \prod_{\ell_r=1}^{\varphi(s_r)}\left(\prod_{X(t_1,\dots,t_r)}\left(1+\prod_{j=1}^r \eta^{\ell_j}_j\right)\right)^{\omega_{L}}= \\
    &= 2^{\left(\sum_{\ell_1=1}^{\varphi(s_1)}\cdots \sum_{\ell_r=1}^{\varphi(s_r)}\gcd(T,L)\omega_{L}\right)} \pmod p.
    \end{align*}

\end{proof}

\subsection{Proof of Corollary 1.14}

\begin{proof}
    Taking norms on both sides of the congruence $f(\eta_1,\eta_2)\equiv d(n)(\text{mod}\ p\mathcal{O}_K)$, which is valid due to Theorem \ref{MainTheoremMultiple}, we get \begin{align*}
    &d(n)^{[K:\mathbb{Q}]}=N_{K/\mathbb{Q}}(d(n))\equiv N_{K/\mathbb{Q}}(f(\eta_1,\eta_2))=\\&=\prod_{\ell_1=1}^{\varphi(s_1)}\prod_{\ell_2=1}^{\varphi(s_2)}\left(N_{K/\mathbb{Q}}\left(1+\eta^{\ell_1}_1\eta^{\ell_2}_2\right)\right)^{\omega_{L}}(\text{mod}\ p).
    \end{align*}
    We have that $\omega_L=0$
     whenever $\nu_2(\ell_1)<\nu_2(\varphi(s_1))$ or $\nu_2(\ell_2)<\nu_2(\varphi(s_2))$. Additionally,  
     $N_{K/\mathbb{Q}}(1+\zeta_m)=1$ provided that $m>1$ is an odd divisor of $\varphi(s)$  and $\zeta_m$ is a primitive $m$th root of unity. Hence, $$2^{[K:\mathbb{Q}]\omega(n)}=d(n)^{[K:\mathbb{Q}]}\equiv (N_{K/\mathbb{Q}}(2))^{S}=2^{[K:\mathbb{Q}]S}(\text{mod}\ p),$$ where $$S=\sum_{\left\{(\ell_1,\ell_2)\ |\  \eta^{\ell_1}_1\eta^{\ell_2}_2=1\right\}}\omega_{L}.$$ Note that one may choose $\eta_1=\zeta^{\varphi(s_2)}$ and $\eta_2=\zeta^{\varphi(s_1)}.$ Clearly, the condition $\eta^{\ell_1}_1\eta^{\ell_2}_2=1$ is equivalent to the condition $$\varphi(s)|\varphi(s_2)\ell_1+\varphi(s_1)\ell_2,$$ which in turn is equivalent to the existence of $k\in\mathbb{Z}$ such that $$\ell_1\equiv \frac{\varphi(s_1)}{d}k \ (\text{mod}\ \varphi(s_1)) \quad \text{and} \quad\ell_2\equiv -\frac{\varphi(s_2)}{d}k \ (\text{mod}\ \varphi(s_2)).$$ In terms of these elements, the congruences can be restated as $$x^k_0\equiv g^{\ell_i}_i(\text{mod}\ s_i), i=1,2.$$ Since the order of $x_0$ modulo $s$ is $d$, it follows that $$S=\sum_{k=0}^{d-1}\omega_{x^k_0}.$$ To conclude the proof, we simply note that
    $[K:\mathbb{Q}]=\varphi(\varphi(s))$.
\end{proof}

\begingroup
\raggedright

\end{document}